\documentclass[11pt]{article}
\usepackage[utf8]{inputenc}
\usepackage[margin=2.5cm]{geometry}
\usepackage{amsmath,amssymb,amsthm}
\usepackage{booktabs,array,enumitem}
\usepackage{tikz}
\usepackage[hidelinks]{hyperref}

\newtheorem{theorem}{Theorem}[section]
\newtheorem{lemma}[theorem]{Lemma}
\newtheorem{proposition}[theorem]{Proposition}

\theoremstyle{definition}
\newtheorem{definition}[theorem]{Definition}
\newtheorem{remark}[theorem]{Remark}

\newcommand{\cV}{\mathcal{V}}
\newcommand{\cL}{\mathcal{L}}
\newcommand{\cR}{\mathcal{R}}
\newcommand{\cD}{\mathcal{D}}
\newcommand{\ox}{\otimes}

\title{The domination number of the $2$-token graph of path graphs}

\author{E. Acosta Troncoso\thanks{Unidad Acad\'emica de Matem\'aticas, Universidad Aut\'onoma de Zacatecas,
Calzada Solidaridad y Paseo La Bufa, Col. Hidr\'aulica, C.P. 98060, Zacatecas, Zac., Mexico.
E-mail: emmanuel\_acosta22@yahoo.com, luismanuel.rivera@gmail.com (corresponding author).}
\and L.~M.~Rivera\footnotemark[1]}
\date{}

\begin{document}
\maketitle

\begin{abstract}
We prove that the domination number of the $2$-token graph of the path $P_n$ is
$\gamma(F_2(P_n))=d(n)$ for every $n\ge 13$, where $d(n)=\frac1{10}(n^2+5n+c)$ and
$c$ is an explicit constant that depends on $n \bmod 5$. This settles a conjecture by Lea\~nos and the authors, who previously proved the upper bound. The lower bound is computer-assisted and follows the method used by
Gon\c{c}alves, Pinlou, Rao and Thomass\'e for grid graphs.
\end{abstract}


\section{Introduction}

Let \(G\) be a simple graph of order \(n \ge 2\). The \emph{\(2\)-token graph} \(F_2(G)\) of \(G\) is the graph whose vertices are the \(2\)-subsets of \(V(G)\), where two vertices are adjacent if and only if their symmetric difference is an edge in \(G\) (see~\cite{FFHH}). The structural properties of \(2\)-token graphs, as well as their generalization to \(k\)-token graphs, have been studied under various names, such as \emph{double vertex graphs}~\cite{alavi2,alavi3} and \emph{symmetric powers of graphs}~\cite{aude}. In recent years, interest in these structures has grown significantly following the work of Fabila-Monroy et al.~\cite{FFHH} (see, e.g.,~\cite{alba,dalfo10,FLT,GR}).

 A set $D\subseteq V(G)$ is
\emph{dominating} if every vertex not in $D$ has a neighbor in $D$, and the \emph{domination number}
$\gamma(G)$ is the minimum size of a dominating set of $G$.

Domination in token graphs has been studied in~\cite{sergio, prabhu}. For the path $P_n$, the graph
$F_2(P_n)$ is a triangular piece of the square grid, and the problem is close to the determination of the
domination number of grid graphs. That problem has a long history: the exact values for grids with at most six rows were obtained
in~\cite{ChangClark}; Chang~\cite{ChangThesis} conjectured the general formula; dynamic programming
methods were developed in~\cite{fisher,HHH} (see also~\cite{alanko,chang} and the references therein); and
Guichard~\cite{guichard} proved the conjectured lower bound up to an additive constant. The problem was
completed by Gon\c{c}alves, Pinlou, Rao and Thomass\'e~\cite{goncalves}, who
proved Chang's conjecture: $\gamma(P_m\square P_n)=\lfloor (n+2)(m+2)/5\rfloor-4$ for $m\ge n\ge 16$. Their
proof is computer-assisted, and their method was later applied to the $2$-domination and Roman domination
numbers of grids~\cite{RaoTalon}.

In~\cite{ALR} it was shown that $\gamma(F_2(P_n))=\frac{n^2}{10}+\Theta(n)$. More precisely, let
\[
d(n):=\begin{cases}
 \frac{1}{10}(n^2+5n-30) & n \equiv 0 \pmod 5, \\
  \frac{1}{10}(n^2+5n-36)  & n \equiv 1 \pmod 5 \; \textup{or } n \equiv 4 \pmod 5,\\
\frac{1}{10}(n^2+5n-34) & n \equiv 2\pmod 5  \;  \textup{or } n \equiv 3 \pmod 5.
\end{cases}
\]
Throughout the paper we write $r=n \bmod 5$ and $d(n)=\frac{1}{10}(n^2+5n+c_r)$; that is,
$(c_0,c_1,c_2,c_3,c_4)=(-30,-36,-34,-34,-36)$.

\begin{theorem}[\cite{ALR}]\label{thm:upper}
If $n\ge 13$, then $\gamma(F_2(P_n))\le d(n)$.
\end{theorem}

The lower bound in~\cite{ALR} comes from the packing number~\cite{gomez} and differs from $d(n)$ by about
$3n/5$. It was conjectured in~\cite{ALR} that $\gamma(F_2(P_n))=d(n)$ for all $n\ge13$, the exact values
being known for $n\le25$. In this paper we prove this conjecture by means of the following:

\begin{theorem}\label{thm:lower}
If $n\ge 13$, then $\gamma(F_2(P_n))\ge d(n)$.
\end{theorem}

The proof of Theorem~\ref{thm:lower} is computer-assisted and is an adaptation of the method used by
Gon\c{c}alves, Pinlou, Rao and Thomass\'e for grid graphs.

The rest of the paper is organized as follows. Section~\ref{sec:prelim} contains the definitions and notation used throughout the paper.
Section~\ref{sec:transfer} introduces the two transfer matrices. Section~\ref{sec:glue} decomposes the band
into four pieces and expresses $b$ as a $(\min,+)$ expression. Theorem~\ref{thm:lower} is proved in Section~\ref{sec:lower}.

\section{Preliminaries}\label{sec:prelim}

Let $I(n)$ be the graph with vertex set $V=\{(x,y)\in\mathbb{Z}^2: 1\le x\le y\le n-1\}$, where
two vertices are adjacent if they are at Euclidean distance one. The map $(x,y)\mapsto\{x,y+1\}$ is an
isomorphism from $I(n)$ to $F_2(P_n)$ (Figure~\ref{fig:I6}). For $S\subseteq V$ let $N[S]$ be its closed neighborhood and
$I(S)=\{v\in S: N[v]\subseteq S\}$ its \emph{interior}. Following~\cite{goncalves}, the \emph{loss} of $S$
is
\[
\ell(S)=5|S|-|N[S]|.
\]
For $B\subseteq V$ we set
\[
b(B)=\min\{\ell(D):\ D\subseteq B,\ I(B)\subseteq N[D]\}.
\]
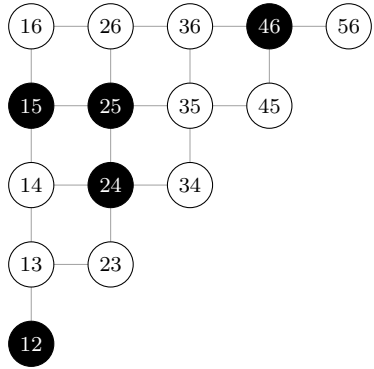
\begin{figure}[ht]
\centering
\begin{tikzpicture}[scale=1.05,font=\scriptsize]
\draw[gray!70] (1,1)--(1,2);
\draw[gray!70] (1,2)--(2,2);
\draw[gray!70] (1,2)--(1,3);
\draw[gray!70] (1,3)--(2,3);
\draw[gray!70] (1,3)--(1,4);
\draw[gray!70] (1,4)--(2,4);
\draw[gray!70] (1,4)--(1,5);
\draw[gray!70] (1,5)--(2,5);
\draw[gray!70] (2,2)--(2,3);
\draw[gray!70] (2,3)--(3,3);
\draw[gray!70] (2,3)--(2,4);
\draw[gray!70] (2,4)--(3,4);
\draw[gray!70] (2,4)--(2,5);
\draw[gray!70] (2,5)--(3,5);
\draw[gray!70] (3,3)--(3,4);
\draw[gray!70] (3,4)--(4,4);
\draw[gray!70] (3,4)--(3,5);
\draw[gray!70] (3,5)--(4,5);
\draw[gray!70] (4,4)--(4,5);
\draw[gray!70] (4,5)--(5,5);
\node[draw,circle,inner sep=1.2pt,minimum size=17pt,fill=black,text=white] at (1,1) {$12$};
\node[draw,circle,inner sep=1.2pt,minimum size=17pt,fill=white] at (1,2) {$13$};
\node[draw,circle,inner sep=1.2pt,minimum size=17pt,fill=white] at (1,3) {$14$};
\node[draw,circle,inner sep=1.2pt,minimum size=17pt,fill=black,text=white] at (1,4) {$15$};
\node[draw,circle,inner sep=1.2pt,minimum size=17pt,fill=white] at (1,5) {$16$};
\node[draw,circle,inner sep=1.2pt,minimum size=17pt,fill=white] at (2,2) {$23$};
\node[draw,circle,inner sep=1.2pt,minimum size=17pt,fill=black,text=white] at (2,3) {$24$};
\node[draw,circle,inner sep=1.2pt,minimum size=17pt,fill=black,text=white] at (2,4) {$25$};
\node[draw,circle,inner sep=1.2pt,minimum size=17pt,fill=white] at (2,5) {$26$};
\node[draw,circle,inner sep=1.2pt,minimum size=17pt,fill=white] at (3,3) {$34$};
\node[draw,circle,inner sep=1.2pt,minimum size=17pt,fill=white] at (3,4) {$35$};
\node[draw,circle,inner sep=1.2pt,minimum size=17pt,fill=white] at (3,5) {$36$};
\node[draw,circle,inner sep=1.2pt,minimum size=17pt,fill=white] at (4,4) {$45$};
\node[draw,circle,inner sep=1.2pt,minimum size=17pt,fill=black,text=white] at (4,5) {$46$};
\node[draw,circle,inner sep=1.2pt,minimum size=17pt,fill=white] at (5,5) {$56$};
\end{tikzpicture}
\caption{The graph $I(6)\cong F_2(P_6)$; the vertex $(x,y)$ is labeled with the $2$-set $\{x,y+1\}$. The black vertices form a dominating set $S$ of size $5$, that is equal to $\gamma(F_2(P_6))$, with $\ell(S)=5\cdot5-15=10$.}\label{fig:I6}
\end{figure}

If $D$ is a dominating set of $I(n)$, then $\ell(D)=5|D|-\binom n2$, and it is easy to see that
$\ell(D)\ge b(B)$ for every $B$ (Lemma~\ref{lem:mono}). Hence
\begin{equation}\label{eq:basic}
5\,\gamma(F_2(P_n))\ \ge\ \tbinom n2+b(B)\qquad\text{for every } B\subseteq V .
\end{equation}
For integers $k\ge2$ and $e\ge1$, the \emph{band} is
\[
B_n(k,e)=\{(x,y)\in V:\ x\le k\ \text{ or }\ y\ge n-k\ \text{ or }\ y-x\le e\},
\]
a border of thickness $k$ along the two straight sides of the triangle and of thickness $e+1$ along the
diagonal side. Let
\[
B'_n=B_n(10,6)\cup\{(11,n-11)\}\quad \text{ and } \quad \nu(n)=3n+\frac{c_r}{2}-4 .
\]
See Figure~\ref{fig:band}.
\begin{figure}[ht]
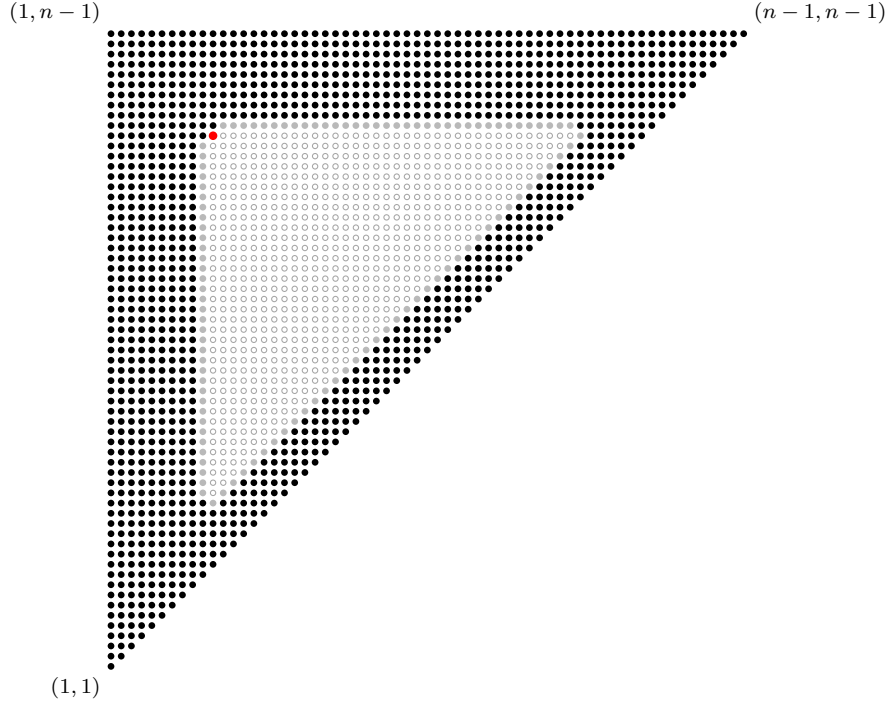

\centering

\caption{The graph $I(64)$ and the band $B'_{64}$. Black: vertices of $I(B'_{64})$; gray: the other vertices of $B'_{64}$; red: the extra vertex $(11,n-11)$; white: vertices outside the band.}\label{fig:band}
\end{figure}

By equation~\eqref{eq:basic}, $\gamma(F_2(P_n))\ge d(n)$ holds as soon as $b(B)\ge\nu(n)$ for some $B$
(Lemma~\ref{lem:target}). The proof of the following result is given in Section~\ref{sec:periodic}. 

\begin{theorem}[computer-assisted]\label{thm:periodic}
For every $n\ge 135$,
\[
b(B'_n)=\begin{cases}
3n-19 & \text{if } n\equiv 0 \pmod 5,\\
3n-20 & \text{if } n\equiv 1,\ 2 \text{ or } 4 \pmod 5,\\
3n-18 & \text{if } n\equiv 3 \pmod 5.
\end{cases}
\]
In particular $b(B'_n)\ge\nu(n)$, with equality if and only if $n\equiv0\pmod5$.
\end{theorem}

The threshold $135$ is the point from which the matrices involved are exactly periodic; it plays the
role of the value $126$ in~\cite[Fact 9]{goncalves}. For $13\le n\le33$ we compute $\gamma(F_2(P_n))$ exactly.


Every vertex of $I(n)$ has degree at most $4$. The vertices $(1,y)$ and $(x,n-1)$ form the two
\emph{straight sides} and the vertices $(x,x)$ the \emph{diagonal side}; a diagonal vertex $(x,x)$ with
$1<x<n-1$ has degree $2$. The reflection $\sigma(x,y)=(n-y,n-x)$ is an automorphism of $I(n)$ that
exchanges the two straight sides. We write $d=y-x$ for the \emph{diagonal index} of $(x,y)$.

\begin{proposition}\label{prop:loss}
For $S,S_1,S_2\subseteq V$:
\begin{enumerate}[label=\textup{(\roman*)},leftmargin=*]
\item $\ell(S)=\sum_{v\in S}(4-\deg v)+\sum_{w\in N[S]}\bigl(|N[w]\cap S|-1\bigr)\ge0$;
\item if $S_1\cap S_2=\emptyset$, then $\ell(S_1\cup S_2)=\ell(S_1)+\ell(S_2)+|N[S_1]\cap N[S_2]|$;
\item if $S'\subseteq S$, then $\ell(S')\le\ell(S)$;
\item $\ell(\sigma(S))=\ell(S)$.
\end{enumerate}
\end{proposition}
\begin{proof}
(i) follows from $\sum_{v\in S}|N[v]|=\sum_{w\in N[S]}|N[w]\cap S|$; (ii) from
$|N[S_1\cup S_2]|=|N[S_1]|+|N[S_2]|-|N[S_1]\cap N[S_2]|$; (iii) from (i) and (ii); (iv) is clear.
\end{proof}

Let $\ell_n$ be the minimum of $\ell(D)$ over all dominating sets $D$ of $I(n)$. Since $N[D]=V$,
\begin{equation}\label{eq:ln}
\ell_n=5\,\gamma(F_2(P_n))-\tbinom n2 .
\end{equation}

\begin{lemma}\label{lem:mono}
If $B\subseteq B'\subseteq V$, then $b(B)\le b(B')$. In particular $b(B)\le b(V)=\ell_n$.
\end{lemma}
\begin{proof}
Let $D'\subseteq B'$ dominate $I(B')$ with $\ell(D')=b(B')$ and let $D=D'\cap B$. If $v\in I(B)$, then
$v\in I(B')$ and it is dominated by some $u\in D'\cap N[v]\subseteq D$. Thus $D$ dominates $I(B)$ and
$b(B)\le\ell(D)\le\ell(D')$ by Proposition~\ref{prop:loss}(iii).
\end{proof}

\begin{lemma}\label{lem:target}
$\gamma(F_2(P_n))\ge d(n)$ if and only if $\ell_n\ge\nu(n)=3n+\frac{c_r}2-4$. In particular it suffices
that $b(B)\ge 3n-19$, $3n-22$, $3n-21$, $3n-21$, $3n-22$ for $n\equiv0,1,2,3,4\pmod 5$, for some $B$.
\end{lemma}
\begin{proof}
A direct computation gives $5d(n)-\binom n2=3n+\frac{c_r}{2}$, so by \eqref{eq:ln},
$\gamma(F_2(P_n))\ge d(n)$ if and only if $\ell_n\ge 3n+\frac{c_r}{2}$. Again by \eqref{eq:ln}, both
$\ell_n$ and $3n+\frac{c_r}{2}$ are congruent to $-\binom n2$ modulo $5$; hence
$\ell_n\ge 3n+\frac{c_r}{2}$ if and only if $\ell_n\ge 3n+\frac{c_r}{2}-4$. The last assertion follows from Lemma~\ref{lem:mono}.
\end{proof}

For $S\subseteq V$ we use the labeling $\phi_S:V\to\{0,1,2\}$ of~\cite{goncalves}: 
\[
\phi_S(v) = \begin{cases}
0 & \text{if } v \in S, \\
1 & \text{if } v \in N[S] \setminus S, \\
2 & \text{otherwise.}
\end{cases}
\]
If $\tau=(\tau_1,\dots,\tau_t)$ is a path of $I(n)$, then $\phi_S(\tau)$ is a word of length $t$ over
$\{0,1,2\}$ with no factor $02$ or $20$. Let $W_t$ be the set of such words; then~(see \cite{fisher, goncalves})
\[
|W_t|=\frac12\bigl((1+\sqrt2)^{t+1}+(1-\sqrt2)^{t+1}\bigr).
\]
So $|W_{10}|=8119$ and $|W_7|=577$. The \emph{Iverson bracket} notation converts a logical proposition $P$ into a binary value and is defined as:
\[
[P] = \begin{cases} 
1 & \text{if } P \text{ is true,} \\ 
0 & \text{if } P \text{ is false.} 
\end{cases}
\]
For a word $w$, $|w|_i$ is the number of occurrences of the letter $i$. Matrices have entries in $\mathbb{Z}\cup\{+\infty\}$, and $\ox$ denotes the matrix product in
the $(\min,+)$ algebra: $(X\ox Y)[i,j]=\min_h X[i,h]+Y[h,j]$.

\section{Transfer matrices}\label{sec:transfer}

\subsection{The straight strip}

Fix $k\ge2$. The \emph{row} $y$ of the left strip is $\rho_y=(\rho_y^1,\dots,\rho_y^k)$ with
$\rho_y^x=(x,y)$. Away from the corners the neighbors of $\rho_y^x$ are $\rho_y^{x\pm1}$,
$\rho_{y\pm1}^x$ and, for $x=k$, the vertex $(k+1,y)$, which has no other neighbor in the strip. Words of
$W_s:=W_k$ are indexed by $x=1,\dots,k$. See Figure~\ref{fig:rows}.

\begin{figure}[ht]
\centering
\begin{tikzpicture}[scale=1.0,font=\scriptsize]
\draw[very thick] (0.4,-0.6)--(0.4,2.6);
\foreach \y in {0,1,2}{ \draw[gray!70] (1,\y)--(4,\y); \draw[dashed] (4,\y)--(5,\y); }
\foreach \x in {1,2,3,4}{ \draw[gray!70] (\x,1)--(\x,2); \draw[gray!40] (\x,-0.5)--(\x,0); \draw[gray!40] (\x,2)--(\x,2.5);}
\foreach \x in {1,2,3,4}{ \draw[line width=1.4pt] (\x,0)--(\x,1); }
\foreach \y in {0,1,2}{ \foreach \x in {1,2,3,4}{ \fill (\x,\y) circle (2.2pt); } \node[draw,fill=white,inner sep=2.2pt] at (5,\y) {}; }
\foreach \x in {1,2,3,4}{ \node[below right] at (\x,0) {$\rho_y^{\x}$}; \node[above right] at (\x,1) {$\rho_{y+1}^{\x}$}; }
\node[right] at (5.2,0) {$(k+1,y)$};
\node[right] at (5.2,1) {$(k+1,y+1)$};
\node[right] at (5.2,2) {row $y+2$};
\node[left] at (0.4,1) {side $x=1$};
\end{tikzpicture}
\caption{Rows $y$, $y+1$, $y+2$ of the left strip for $k=4$. Thick edges: $\rho_y^x\rho_{y+1}^x$. Squares:
vertices of column $k+1$, outside the strip.}\label{fig:rows}
\end{figure}
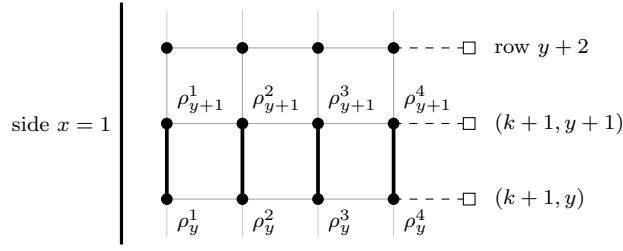

\begin{definition}\label{def:Ts}
For $w,w'\in W_s$ let
\[
\Delta_s(w,w')=4\,|w'|_0-\bigl|\{x: w_x=2,\ w'_x=0\}\bigr|-\bigl|\{x: w'_x\ne2,\ w_x\ne0\}\bigr|-[w'_k=0],
\]
and define the $|W_s|\times|W_s|$ matrix $T_s$ by
\[
T_s[w,w']=\left\{
\setlength{\arraycolsep}{2.5pt}\begin{array}{lll}
+\infty & \text{if } \exists\,x,\ 1\le x\le k, \text{ s.t. } w_x=0 \text{ and } w'_x=2, & \textup{(Q1)}\\[2pt]
+\infty & \text{if } \exists\,x,\ 1\le x\le k, \text{ s.t. } w'_x=1,\ w_x\ne0,\ w'_{x-1}\ne0 \text{ and } w'_{x+1}\ne0, & \textup{(Q2)}\\[2pt]
+\infty & \text{if } \exists\,x,\ 1\le x\le k-1, \text{ s.t. } w_x=2 \text{ and } w'_x\ne0, & \textup{(Q3)}\\[2pt]
\Delta_s(w,w') & \text{otherwise,}
\end{array}\right.
\]
with the convention $w'_0=w'_{k+1}=2$.
\end{definition}

\begin{lemma}\label{lem:Ts}
Let $k\le y\le n-3$, let $S\subseteq V$ be a set all of whose vertices lie in rows $\le y$ with
$(k+1,y)\notin S$, and let $S'=S\cup A$ with $A\subseteq\rho_{y+1}$. Put $w=\phi_S(\rho_y)$ and
$w'=\phi_{S'}(\rho_{y+1})$. Then \textup{(Q1)} and \textup{(Q2)} do not hold, and
\[
\ell(S')-\ell(S)=\Delta_s(w,w').
\]
Moreover, the vertices $\rho_y^x$ with $x\le k-1$ are all dominated by $S'$ if and only if
\textup{(Q3)} does not hold.
\end{lemma}
\begin{proof}
We have $\ell(S')-\ell(S)=5|A|-|N[S']\setminus N[S]|$ and $|A|=|w'|_0$. The set $N[S']\setminus N[S]$
consists of: the vertices $\rho_y^x$ with $w_x=2$ and $w'_x=0$; the vertices $\rho_{y+1}^x$ with
$w'_x\ne2$ that have no neighbor in $S$, i.e.\ with $w_x\ne0$; the $|w'|_0$ vertices $\rho_{y+2}^x$ with
$w'_x=0$; and the vertex $(k+1,y+1)$ if $w'_k=0$. This gives the formula. Conditions \textup{(Q1)} and
\textup{(Q2)} only express that $w'$ is the word of $S'$: a label $2$ forbids the neighbor below to be
in $S$, and a label $1$ requires a dominator, which is $\rho_{y+1}^{x\pm1}$ or $\rho_y^x$. Finally, the
only neighbor of $\rho_y^x$ in $S'\setminus S$ is $\rho_{y+1}^x$.
\end{proof}

By symmetry, the same matrix serves for the top strip. For $k=10$, $T_s$ has $2\,923\,376$ finite
entries out of $8119^2$.

\begin{remark}\label{rem:gprt}
For $k=10$ the strip is the one of~\cite{goncalves}, and using \textup{(Q1)}--\textup{(Q3)}, $\Delta_s(w,w')$ can be rewritten as
$3|w'|_0-|w|_2-|w'|_1+|w|_0-[w'_k=0]+[w_k=2][w'_k\ne0]$. The formula printed in~\cite[Lemma 8]{goncalves}
is the same without the last term: the vertex $\rho_y^k$ is not interior and may remain undominated, in
which case it must not be counted. 
\end{remark}

\subsection{The diagonal strip}

Fix $e\ge1$. The \emph{column} $t$ of the diagonal strip is $\kappa_t=(\kappa_t^0,\dots,\kappa_t^e)$
with $\kappa_t^d=(t,t+d)$, and $o_t=(t,t+e+1)$ is the associated \emph{exterior vertex}
(Figure~\ref{fig:columns}). Away from the corners,
\[
N(\kappa_t^d)=\{\kappa_t^{d-1},\kappa_t^{d+1},\kappa_{t-1}^{d+1},\kappa_{t+1}^{d-1}\}\ \
(\text{indices in } [0,e]),\qquad\text{plus } o_{t-1},o_t \text{ if } d=e .
\]
Thus consecutive columns are joined by the edges $\kappa_t^d\kappa_{t+1}^{d-1}$, $1\le d\le e$, the
diagonal vertex $\kappa_t^0$ has degree $2$, and $o_t$ is adjacent to both $\kappa_t^e$ and
$\kappa_{t+1}^e$. Words of $W_d:=W_{e+1}$ are indexed by $d=0,\dots,e$.

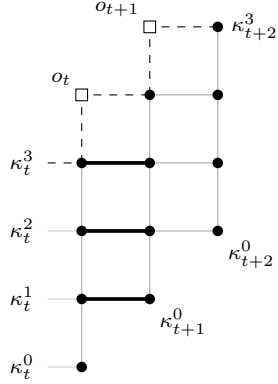
\begin{figure}[ht]
\centering
\begin{tikzpicture}[scale=0.9,font=\scriptsize]
\foreach \t in {0,1,2}{ \foreach \d in {0,1,2}{ \draw[gray!70] (\t,\t+\d)--(\t,\t+\d+1); } }
\foreach \d in {1,2,3}{ \draw[line width=1.4pt] (0,\d)--(1,\d); }
\foreach \d in {1,2,3}{ \draw[gray!70] (1,1+\d)--(2,1+\d); \draw[gray!40] (-0.5,\d-1)--(0,\d-1);}
\draw[dashed] (0,3)--(0,4)--(1,4);
\draw[dashed] (1,4)--(1,5)--(2,5);
\draw[dashed] (-0.5,3)--(0,3);
\node[draw,fill=white,inner sep=2.2pt] at (0,4) {};
\node[draw,fill=white,inner sep=2.2pt] at (1,5) {};
\node[above left] at (0,4) {$o_t$};
\node[above left] at (1,5) {$o_{t+1}$};
\foreach \t in {0,1,2}{ \foreach \d in {0,1,2,3}{ \fill (\t,\t+\d) circle (2.2pt); } }
\foreach \d in {0,1,2,3}{ \node[left] at (-0.55,\d) {$\kappa_t^{\d}$}; }
\node[below right] at (1,1) {$\kappa_{t+1}^{0}$};
\node[below right] at (2,2) {$\kappa_{t+2}^{0}$};
\node[right] at (2.05,5) {$\kappa_{t+2}^{3}$};
\end{tikzpicture}
\caption{Columns $t$, $t+1$, $t+2$ of the diagonal strip for $e=3$. Thick edges:
$\kappa_t^d\kappa_{t+1}^{d-1}$. Squares: exterior vertices.}\label{fig:columns}
\end{figure}

\begin{definition}\label{def:Td}
For $w,w'\in W_d$ let
\begin{multline*}
\Delta_d(w,w')=5\,|w'|_0-\bigl|\{1\le d\le e: w_d=2,\ w'_{d-1}=0\}\bigr|
-\bigl|\{0\le d\le e-1: w'_d\ne2,\ w_{d+1}\ne0\}\bigr|\\
-[w'_e\ne2]-\bigl|\{1\le d\le e: w'_d=0\}\bigr|-[w'_e=0]\bigl(1+[w_e\ne0]\bigr),
\end{multline*}
and define the $|W_d|\times|W_d|$ matrix $T_d$ by
\[
T_d[w,w']=\left\{
\setlength{\arraycolsep}{2.5pt}\begin{array}{lll}
+\infty & \text{if } w'_0=2, & \textup{(P1)}\\[2pt]
+\infty & \text{if } \exists\,d,\ 0\le d\le e-1, \text{ s.t. } w'_d=2 \text{ and } w_{d+1}=0, & \textup{(P2)}\\[2pt]
+\infty & \text{if } \exists\,d,\ 0\le d\le e, \text{ s.t. } w'_d=1,\ w'_{d-1}\ne0,\ w'_{d+1}\ne0 \text{ and } w_{d+1}\ne0, & \textup{(P3)}\\[2pt]
+\infty & \text{if } \exists\,d,\ 1\le d\le e-1, \text{ s.t. } w_d=2 \text{ and } w'_{d-1}\ne0, & \textup{(P4)}\\[2pt]
\Delta_d(w,w') & \text{otherwise,}
\end{array}\right.
\]
with the convention $w'_{-1}=w'_{e+1}=w_{e+1}=2$.
\end{definition}

\begin{lemma}\label{lem:Td}
Let $1\le t\le n-e-3$, let $S\subseteq\{(x,y)\in V: y-x\le e,\ x\le t\}$, and let $S'=S\cup A$ with
$A\subseteq\kappa_{t+1}$. Put $w=\phi_S(\kappa_t)$ and $w'=\phi_{S'}(\kappa_{t+1})$. Then \textup{(P2)} and
\textup{(P3)} do not hold and
\[
\ell(S')-\ell(S)=\Delta_d(w,w').
\]
Moreover, $\kappa_{t+1}^0$ and the vertices $\kappa_t^d$ with $1\le d\le e-1$ are all dominated by $S'$
if and only if neither \textup{(P1)} nor \textup{(P4)} holds.
\end{lemma}
\begin{proof}
As in Lemma~\ref{lem:Ts}. The set $N[S']\setminus N[S]$ consists of: the $\kappa_t^d$ with $w_d=2$ and
$\kappa_{t+1}^{d-1}\in A$; the $\kappa_{t+1}^d$ with $w'_d\ne2$ and no neighbor in $S$, i.e.\ $d=e$ or
$w_{d+1}\ne0$; the $\kappa_{t+2}^{d-1}$ with $d\ge1$ and $w'_d=0$; the vertex $o_t$ if $w'_e=0$ and
$w_e\ne0$; and $o_{t+1}$ if $w'_e=0$. The two neighbors of $\kappa_{t+1}^0$ are $\kappa_t^1$ and
$\kappa_{t+1}^1$, and the only neighbor of $\kappa_t^d$ in $A$ is $\kappa_{t+1}^{d-1}$.
\end{proof}

\section{Decomposition of the band}\label{sec:glue}

The exact value of $b$ on a band requires coupling the three strips at the corners. We fix integers with
\begin{equation}\label{eq:hyp}
k\ge2,\qquad e\ge1,\qquad a\ge k+e+2,\qquad n\ge 2a+3, \tag{$\star$}
\end{equation}
write $B=B_n(k,e)$, and set $p=n-a-k-2$ and $m=n-2a-1$.

\begin{definition}\label{def:pieces}
The \emph{pieces} of $B$ are (Figure~\ref{fig:pieces})
\begin{align*}
\Pi_p&=\cV\cup\cL_p,\quad\text{where } \cV=B\cap\{y\le a\} \text{ and }\cL_p=\{x\le k,\ a<y\le a+p\};\\
\cR&=B\cap\{x\le k+1,\ y\ge n-k-1\};\\
\qquad \Pi'_p&=\sigma(\Pi_p);\\
\cD_m&=\{y-x\le e,\ y\ge a+1,\ x\le n-a-1\}.
\end{align*}
The \emph{slices} are 
\[
\alpha=(\alpha_x)_{x=1}^k, \beta=(\beta_x)_{x=1}^k, \delta=(\delta_d)_{d=0}^e, \varepsilon=(\varepsilon_d)_{d=0}^e
\]
and their images under $\sigma$, where
\[
\alpha_x=(x,n-k-2)\in\Pi_p,\quad \beta_x=(x,n-k-1)\in\cR,\quad
\delta_d=(a-d,a)\in\cV,\quad \varepsilon_d=(a+1-d,a+1)\in\cD_m .
\]
\end{definition}

Thus $\cV$ is the \emph{acute corner} at $(1,1)$, $\cR$ the \emph{right corner} at $(1,n-1)$ (a square of
side $k+1$ minus the vertex $c^*=(k+1,n-k-1)$), and $\cD_m$ consists of a fixed triangle and $m$ full
columns $\kappa_{a+1},\dots,\kappa_{a+m}$, the last one being $\sigma(\varepsilon)$. Both $\cR$ and
$\cD_m$ are invariant under $\sigma$.

\begin{figure}[ht]
\centering
\begin{tikzpicture}[scale=0.62,font=\small]
\fill[blue!14] (0,0)--(1.2,1.2)--(1.2,8.6)--(0,8.6)--cycle;
\fill[blue!14] (0,0)--(3.2,3.2)--(2.3,3.2)--(0,0.9)--cycle;
\fill[red!16] (0,8.6)--(1.2,8.6)--(1.2,8.8)--(1.4,8.8)--(1.4,10)--(0,10)--cycle;
\fill[green!16] (1.4,8.8)--(8.8,8.8)--(10,10)--(1.4,10)--cycle;
\fill[green!16] (6.8,6.8)--(10,10)--(9.1,10)--(6.8,7.7)--cycle;
\fill[orange!25] (3.2,3.2)--(6.8,6.8)--(6.8,7.7)--(2.3,3.2)--cycle;
\draw[thick] (0,0)--(0,10)--(10,10)--cycle;
\draw (1.2,2.1)--(1.2,8.8)--(7.9,8.8)--cycle;
\draw[very thick,dashed] (0,8.6)--(1.2,8.6);
\draw[very thick,dashed] (1.4,8.8)--(1.4,10);
\draw[very thick,dashed] (2.3,3.2)--(3.2,3.2);
\draw[very thick,dashed] (6.8,6.8)--(6.8,7.7);
\draw[dotted,thick] (0,3.2)--(1.2,3.2);
\draw[dotted,thick] (6.8,8.8)--(6.8,10);
\node at (0.6,1.9) {$\cV$};
\node at (0.6,6.0) {$\cL_p$};
\node at (0.62,9.4) {$\cR$};
\node at (4.1,9.4) {$\sigma(\cL_p)$};
\node at (8.2,9.4) {$\sigma(\cV)$};
\node at (5.75,4.75) {$\cD_m$};
\node[blue!60!black] at (-0.8,5.2) {$\Pi_p$};
\node[green!40!black] at (5.5,11.0) {$\Pi'_p$};
\node[below left] at (0,0) {\scriptsize$(1,1)$};
\node[above left] at (0,10) {\scriptsize$(1,n-1)$};
\node[above right] at (10,10) {\scriptsize$(n-1,n-1)$};
\node[left] at (0,8.6) {\scriptsize$\alpha\,|\,\beta$};
\node[left] at (0,3.2) {\scriptsize$y=a$};
\node (c3) at (4.6,2.5) {\scriptsize$\delta\,|\,\varepsilon$};
\draw[->,thin] (c3)--(2.85,3.12);
\end{tikzpicture}
\qquad
\begin{tikzpicture}[scale=0.8,font=\scriptsize]
\foreach \y in {-1,0,1,2}{
  \pgfmathtruncatemacro{\xa}{\y-3}
  \pgfmathtruncatemacro{\xb}{\y-1}
  \foreach \x in {\xa,...,\xb}{ \draw[gray!60] (\x,\y)--(\x+1,\y); }
}
\foreach \y in {-1,1}{
  \pgfmathtruncatemacro{\xa}{\y-2}
  \foreach \x in {\xa,...,\y}{ \draw[gray!60] (\x,\y)--(\x,\y+1); }
}
\foreach \x in {-2,-1,0}{ \draw[line width=1.6pt] (\x,0)--(\x,1); }
\draw[dashed] (-3,0)--(-3,1)--(-2,1);
\node[draw,fill=white,inner sep=2.2pt] at (-3,1) {};
\node[above left] at (-3,1) {$z$};
\foreach \y in {-1,0,1,2}{
  \pgfmathtruncatemacro{\xa}{\y-3}
  \foreach \x in {\xa,...,\y}{ \fill (\x,\y) circle (2.2pt); }
}
\foreach \dd/\x in {0/0,1/-1,2/-2,3/-3}{ \node[below right] at (\x,0) {$\delta_{\dd}$}; }
\foreach \dd/\x in {0/1,1/0,2/-1,3/-2}{ \node[above right] at (\x,1) {$\varepsilon_{\dd}$}; }
\draw[very thick,dashed,red!70!black] (-3.8,0.5)--(1.6,0.5);
\end{tikzpicture}
\caption{Left: the pieces $\Pi_p=\cV\cup\cL_p$, $\cR$, $\Pi'_p$, $\cD_m$ of the band and the four cuts
(not to scale). Right: the diagonal cut for $e=3$; the exterior vertex $z=(a-e,a+1)$ is adjacent to
$\delta_e$ and $\varepsilon_e$.}\label{fig:pieces}
\end{figure}
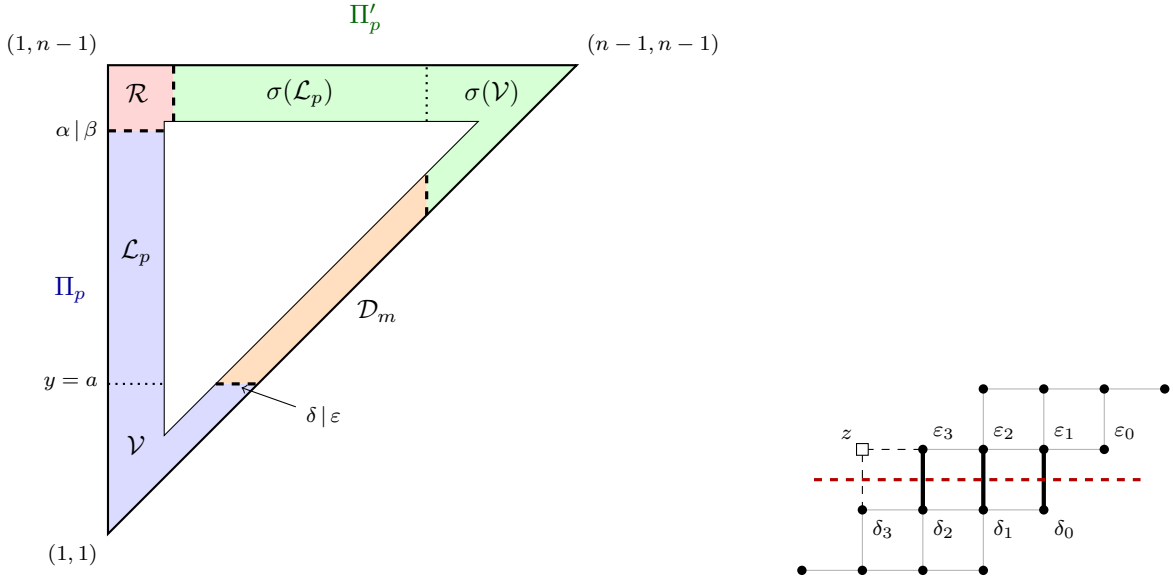

\begin{lemma}\label{lem:glue}
Let $B=P_1\sqcup\dots\sqcup P_t\subseteq V$ be such that every vertex of $I(n)$ belongs to $N[P_i]$ for at
most two indices $i$. Let $D\subseteq B$ and $D_i=D\cap P_i$. Then
\begin{enumerate}[label=\textup{(\alph*)},leftmargin=*]
\item $\ell(D)=\sum_i\ell(D_i)+\sum_{i<j}|N[D_i]\cap N[D_j]|$;
\item $D$ dominates $I(B)$ if and only if each $D_i$ dominates $I(P_i)$ and every
$v\in(I(B)\cap P_i)\setminus I(P_i)$ lies in $N[D_i]$ or has a neighbor in $D\setminus D_i$;
\item for $i\ne j$, $N[D_i]\cap N[D_j]$ is the disjoint union of the following sets
\[
\{a\in P_i\cap N[D_i]: N(a)\cap D_j\ne\emptyset\},
\] 
\[\{b\in P_j\cap N[D_j]: N(b)\cap D_i\ne\emptyset\},
\]
\[
\{z\notin B: N(z)\cap D_i\ne\emptyset\ne N(z)\cap D_j\}.
\]
\end{enumerate}
\end{lemma}
\begin{proof}
(a) By hypothesis, the triple intersections of the sets $N[D_i]$ are empty, so inclusion--exclusion gives
$|N[D]|=\sum_i|N[D_i]|-\sum_{i<j}|N[D_i]\cap N[D_j]|$. (b) If $v\in I(P_i)$, then $N[v]\subseteq P_i$.
(c) A vertex of $N[D_i]\cap N[D_j]$ lies in no third piece; if it lies in $P_i$ it is not in $D_j$ and
hence has a neighbor in $D_j$, and similarly for $P_j$.
\end{proof}

\begin{lemma}\label{lem:structure}
Assume \eqref{eq:hyp}. Then:
\begin{enumerate}[label=\textup{(\arabic*)},leftmargin=*]
\item $B=\Pi_p\sqcup\cR\sqcup\Pi'_p\sqcup\cD_m$;
\item the edges between distinct pieces are $\alpha_x\beta_x$ $(1\le x\le k)$,
$\delta_d\varepsilon_{d+1}$ $(0\le d\le e-1)$ and their images under $\sigma$;
\item $N[\Pi_p]\cap N[\Pi'_p]=\emptyset=N[\cR]\cap N[\cD_m]$, so the hypothesis of Lemma~\ref{lem:glue}
holds;
\item the only vertices outside $B$ with neighbors in two pieces are $z=(a-e,a+1)$, adjacent to
$\delta_e$ and $\varepsilon_e$, and $\sigma(z)$;
\item among the vertices of the slices, those in $I(B)$ are exactly $\alpha_x,\beta_x$ with $x\le k-1$ and
$\delta_d,\varepsilon_d$ with $d\le e-1$ (and their images); each endpoint of a crossing edge has exactly
one neighbor outside its piece.
\end{enumerate}
\end{lemma}
\begin{proof}
A direct verification in coordinates. For (1): a vertex of $B$ with $x\le k$ lies in $\Pi_p$ or $\cR$
according to whether $y\le n-k-2$; one with $x\ge k+1$ and $y\ge n-k$ lies in $\cR$ if $x=k+1$ and in
$\Pi'_p$ otherwise; the remaining ones have $y-x\le e$ and lie in $\cV$, $\sigma(\cV)$ or $\cD_m$
according to whether $y\le a$, $x\ge n-a$ or neither. In $\cD_m$ we have $x\ge a+1-e\ge k+3$ and
$y\le n-k-3$, which gives disjointness and (2). For (3), $N[\Pi_p]$ and $N[\Pi'_p]$ can only meet at
$(k+1,n-k-1)$, whose neighbors in $B$ lie in $\cR$, and in $N[\cR]$ one has $y-x\ge n-2k-4>e+1$. For (4),
a vertex outside $B$ adjacent to $\cD_m$ and $\Pi_p$ has $x\ge k+2$, $y-x=e+1$ and $y\in\{a,a+1\}$, and
only $y=a+1$ is possible. For (5), the neighbors of $\alpha_k,\beta_k$ in column $k+1$ and those of
$\delta_e,\varepsilon_e$ on the diagonal $e+1$ are not in $B$.
\end{proof}

\begin{definition}\label{def:L}
For $w,w'\in W_s$ and $u,v\in W_d$ let
\begin{align*}
\lambda_s(w,w')&=\sum_{x=1}^{k}\bigl([w_x\ne2][w'_x=0]+[w'_x\ne2][w_x=0]\bigr),\\
\lambda_d(u,v)&=\sum_{d=0}^{e-1}\bigl([u_d\ne2][v_{d+1}=0]+[v_{d+1}\ne2][u_d=0]\bigr)+[u_e=0][v_e=0].
\end{align*}
The pair $(w,w')$ is \emph{compatible} if for no $x\le k-1$ we have ($w_x=2$, $w'_x\ne0$) or ($w'_x=2$,
$w_x\ne0$). The pair $(u,v)$ is \emph{compatible} if for no $d\le e-1$ we have ($u_d=2$,
$v_{d+1}\ne0$) and for no $d\le e-2$ we have ($v_{d+1}=2$, $u_d\ne0$). Let $L_s[w,w']=\lambda_s(w,w')$
and $L_d[u,v]=\lambda_d(u,v)$ for compatible pairs, and $+\infty$ otherwise.
\end{definition}

The matrix $L_s$ is the matrix $L$ of~\cite{goncalves}. The matrix $L_d$ is not symmetric (its first
index is the word on the side of the acute corner), and the last term of $\lambda_d$ accounts for the
exterior vertex $z$.

\begin{lemma}\label{lem:cuts}
Let $D\subseteq B$, $D_1=D\cap\Pi_p$, $D_2=D\cap\cR$, $D_4=D\cap\cD_m$, and let $w=\phi_{D_1}(\alpha)$,
$w'=\phi_{D_2}(\beta)$, $u=\phi_{D_1}(\delta)$, $v=\phi_{D_4}(\varepsilon)$. Then
\[
|N[D_1]\cap N[D_2]|=\lambda_s(w,w') \quad \text {and } \quad |N[D_1]\cap N[D_4]|=\lambda_d(u,v).
\]

Moreover, 
\begin{itemize}
\item the vertices of
$\alpha\cup\beta$ lying in $I(B)$ are all dominated by $D$ if and only if $(w,w')$ is compatible,
\item the
endpoints of crossing edges in $\delta\cup\varepsilon$ lying in $I(B)$ are all dominated by $D$ if and
only if $(u,v)$ is compatible. 
\end{itemize} 
The same holds for the images under $\sigma$.
\end{lemma}
\begin{proof}
By Lemma~\ref{lem:glue}(c) and Lemma~\ref{lem:structure}: $\alpha_x\in N[D_1]$ means $w_x\ne2$, and its
only neighbor in $\cR$ is $\beta_x$, which lies in $D_2$ if and only if $w'_x=0$; in the diagonal cut the
vertex $z$ contributes if and only if $\delta_e\in D_1$ and $\varepsilon_e\in D_4$. The second part
follows from Lemma~\ref{lem:structure}(5).
\end{proof}

For $u,v,v'\in W_d$ and $w,w'\in W_s$ define
\begin{align*}
C^{\cV}_p[u,w]&=\min\{\ell(D): D\subseteq\Pi_p,\ I(\Pi_p)\subseteq N[D],\ \phi_D(\delta)=u,\ \phi_D(\alpha)=w\},\\
C^{\cR}[w,w']&=\min\{\ell(D): D\subseteq\cR,\ I(\cR)\subseteq N[D],\ \phi_D(\beta)=w,\ \phi_D(\sigma(\beta))=w'\},\\
C^{\cD}_m[v,v']&=\min\{\ell(D): D\subseteq\cD_m,\ I(\cD_m)\subseteq N[D],\ \phi_D(\varepsilon)=v,\ \phi_D(\sigma(\varepsilon))=v'\},
\end{align*}
with $\min\emptyset=+\infty$. Under \eqref{eq:hyp} these matrices depend only on $(k,e,a,p)$, on $k$,
and on $(e,m)$, respectively; $C^{\cR}$ and $C^{\cD}_m$ are symmetric, and the matrix of $\Pi'_p$ is
$C^{\cV}_p$, by Proposition~\ref{prop:loss}(iv).

\begin{theorem}\label{thm:glue}
Assume \eqref{eq:hyp} and let
\[
H=L_s\ox C^{\cR}\ox L_s,\qquad A_p=C^{\cV}_p\ox H\ox(C^{\cV}_p)^{T},\qquad
G_m=L_d\ox C^{\cD}_m\ox L_d^{T}.
\]
Then
\[
b(B_n(k,e))=\min_{u,u'\in W_d}\bigl(A_p[u,u']+G_m[u,u']\bigr).
\]
\end{theorem}
\begin{proof}
First we prove the side ``$\ge$''. Let $D\subseteq B$ dominate $I(B)$ with $\ell(D)=b(B)$, and let $D_1$, $D_2$, $D_3$, $D_4$ be its traces on
$\Pi_p$, $\cR$, $\Pi'_p$, $\cD_m$. By Lemmas~\ref{lem:glue}(b), \ref{lem:structure}(5) and~\ref{lem:cuts}, each
$D_i$ dominates the interior of its piece and the four pairs of words read at the cuts are compatible.
By Lemmas~\ref{lem:glue}(a) and~\ref{lem:cuts}, $\ell(D)$ is the sum of the four losses $\ell(D_i)$ and of
the four gluing costs, and each $\ell(D_i)$ is at least the corresponding matrix entry. Conversely, given words for which the sum is finite, choose sets $D_i$ attaining the four
minima; by the same lemmas their union dominates $I(B)$ and has loss equal to the sum.
\end{proof}

Only the inequality ``$\ge$'' (the analogue of~\cite[Lemma 7]{goncalves}) is needed for
Theorem~\ref{thm:lower}. The dependence on $n$ is as follows.

\begin{lemma}\label{lem:transfer}
For every $p\ge0$ let $C^{\cV}_p$ be the matrix defined as above for the set $\Pi_p=\cV\cup\cL_p$, with
$\alpha$ read on the row $a+p$, in any triangle $I(N)$ with $N\ge a+p+k+2$ (it does not depend on $N$);
similarly, let $C^{\cD}_m$ be defined for every $m\ge2$. Then
$C^{\cV}_{p+1}=C^{\cV}_p\ox T_s$ for $p\ge0$, and $C^{\cD}_{m+1}=C^{\cD}_m\ox T_d$ for $m\ge2$.
\end{lemma}
\begin{proof}
This is the argument of~\cite[Lemma 8]{goncalves}. If $S'\subseteq\Pi_{p+1}$ dominates $I(\Pi_{p+1})$
and $S=S'\cap\Pi_p$, then $S$ dominates $I(\Pi_p)$, the word on $\delta$ is unchanged, and by
Lemma~\ref{lem:Ts} the increment $\ell(S')-\ell(S)$ is $T_s[w,w']<\infty$, because
$I(\Pi_{p+1})=I(\Pi_p)\cup\{\rho^x_{a+p}:x\le k-1\}$. Conversely, every set counted in
$C^{\cV}_p[u,w]$ extends, for each $w'$ with $T_s[w,w']<\infty$, to a set counted in
$C^{\cV}_{p+1}[u,w']$. The diagonal case is identical, using Lemma~\ref{lem:Td} and
$I(\cD_{m+1})=I(\cD_m)\cup\{\kappa_{a+m}^d:1\le d\le e-1\}\cup\{\kappa_{a+m+1}^0\}$.
\end{proof}

\section{Proof of Theorem~\ref{thm:periodic}}\label{sec:periodic}

Throughout this section $(k,e,a)=(10,6,18)$, so that $p=n-30$ and $m=n-37$.

\subsection{Periodicity}

\begin{proposition}[computer-assisted]\label{prop:perD}
For $e=6$, $C^{\cD}_{m+5}=C^{\cD}_m+5$ for every $m\ge10$. More generally, for $1\le e\le7$ there are
$c,g,m_0$ such that $C^{\cD}_{m+c}=C^{\cD}_m+g$ for all $m\ge m_0$:
\begin{center}
\begin{tabular}{c|ccccccc}
$e$ & $1$ & $2$ & $3$ & $4$ & $5$ & $6$ & $7$\\\hline
$(c,g)$ & $(2,0)$ & $(3,2)$ & $(1,1)$ & $(1,1)$ & $(2,2)$ & $(5,5)$ & $(5,5)$\\
$m_0$ & $4$ & $5$ & $7$ & $6$ & $8$ & $10$ & $14$
\end{tabular}
\end{center}
\end{proposition}

\begin{proposition}[computer-assisted]\label{prop:perS}
For $k=10$, $T_s^{\ox(q+1)}=T_s^{\ox q}+1$ for every $q\ge119$. Consequently, for any $e$, $a$ and any
choice of the acute corner, $C^{\cV}_{p+1}=C^{\cV}_p+1$ for every $p\ge119$. For the corner $\cV$ with
$(k,e,a)=(10,6,18)$ one has, more precisely, $C^{\cV}_{p+1}=C^{\cV}_p+1$ for every $p\ge105$.
\end{proposition}

In both propositions the equality is checked entrywise for the first value of the index and then
propagates, because $(X+g)\ox T=(X\ox T)+g$ and by Lemma~\ref{lem:transfer}. The matrix $C^{\cD}_2$ is
computed column by column with the counting of Lemma~\ref{lem:Td} adapted to the truncated columns of
the initial triangle.

\subsection{The corners and the extra vertex}

The matrices $C^{\cV}_0$ and $H$ are computed by a cell-by-cell dynamic program. The state is a function
on labelings in $\{0,1,2\}$ of the current \emph{frontier} (processed cells having an unprocessed
neighbor, together with the cells of the output slices), with the meaning of $\phi$. When a cell $c$
whose processed neighbors have labels $s_1,\dots,s_j$ is added to the set, the cost is
$5-1-\mathrm{ext}(c)-|\{i:s_i=2\}|$, where $\mathrm{ext}(c)$ is the number of vertices outside the piece
whose only neighbor in the piece is $c$, and the labels $2$ become $1$; when it is not added, it
receives the label $1$ at cost $-1$ if some $s_i=0$, and the label $2$ at cost $0$ otherwise. In this way
every vertex of $N[S]$ is counted exactly once. A cell is closed when all its neighbors are processed,
and a closed interior cell must have label $0$ or $1$. A vertex outside the piece with two or more
neighbors in the piece is processed as a \emph{virtual} cell that cannot be added to the set.

For the acute corner $\cV$ ($170$ vertices and one virtual cell) the cells are processed by increasing
$x+y$, and the frontier never exceeds $17$ cells; $C^{\cV}_0$ is a $577\times8119$ matrix with
$3\,077\,101$ finite entries. For the right corner, the same program is run on $\Pi_p\cup\cR$ starting
from the row $\alpha$, the cells of $\cR$ being processed in order of increasing slope around the inner
corner (frontier at most $12$); applied to the identity matrix this yields $H$.

\medskip
Without modification this band falls short by one when $n\equiv0\pmod5$ (compare the rows $b(B_n(10,6))$ and $\nu(n)$ of
Table~\ref{tab:periodic}). As in~\cite{goncalves}, the remedy is an extra vertex at the right
corner. Let $c^*=(k+1,n-k-1)$, $B'=B\cup\{c^*\}$ and $\cR'=\cR\cup\{c^*\}$, and let
$z_1=(k+1,n-k-2)$ and $z_2=(k+2,n-k-1)$ be the two neighbors of $c^*$ outside $B'$
(Figure~\ref{fig:extra}).

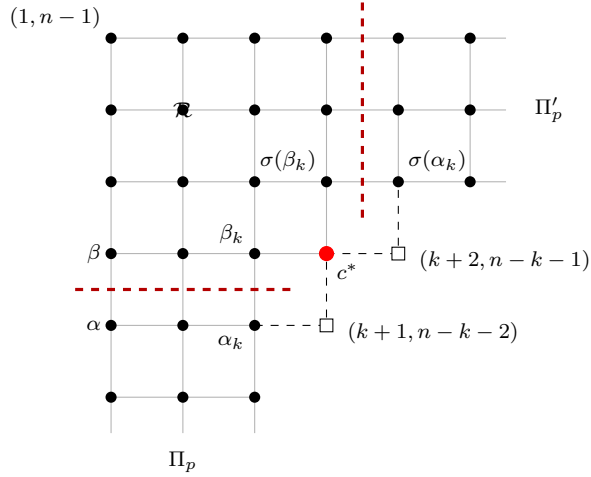
\begin{figure}[ht]
\centering
\begin{tikzpicture}[scale=0.95,font=\scriptsize]
\foreach \Y in {-5,-4,-3}{ \draw[gray!60] (1,\Y)--(3,\Y); }
\foreach \Y in {-2,-1,0}{ \draw[gray!60] (1,\Y)--(6.5,\Y); }
\foreach \x in {1,2,3}{ \draw[gray!60] (\x,-5.5)--(\x,0); }
\foreach \x in {4,5,6}{ \draw[gray!60] (\x,-2)--(\x,0); }
\draw[gray!60] (3,-3)--(4,-3)--(4,-2);
\draw[very thick,dashed,red!70!black] (0.5,-3.5)--(3.5,-3.5);
\draw[very thick,dashed,red!70!black] (4.5,-2.5)--(4.5,0.5);
\draw[dashed] (3,-4)--(4,-4)--(4,-3);
\draw[dashed] (4,-3)--(5,-3)--(5,-2);
\node[draw,fill=white,inner sep=2.4pt] at (4,-4) {};
\node[draw,fill=white,inner sep=2.4pt] at (5,-3) {};
\foreach \Y in {-5,-4,-3}{ \foreach \x in {1,2,3}{ \fill (\x,\Y) circle (2.2pt); } }
\foreach \Y in {-2,-1,0}{ \foreach \x in {1,2,3,4,5,6}{ \fill (\x,\Y) circle (2.2pt); } }
\fill[red] (4,-3) circle (3pt);
\node[below right] at (4,-3) {$c^*$};
\node[left] at (1,-4) {$\alpha$};
\node[left] at (1,-3) {$\beta$};
\node[below left] at (3,-4) {$\alpha_k$};
\node[above left] at (3,-3) {$\beta_k$};
\node[above left] at (4,-2) {$\sigma(\beta_k)$};
\node[above right] at (5,-2) {$\sigma(\alpha_k)$};
\node[right] at (4.15,-4.1) {$(k+1,n-k-2)$};
\node[right] at (5.15,-3.1) {$(k+2,n-k-1)$};
\node[above left] at (1,0) {$(1,n-1)$};
\node at (2,-1) {$\cR$};
\node at (2,-5.9) {$\Pi_p$};
\node at (7.1,-1) {$\Pi'_p$};
\end{tikzpicture}
\caption{The right corner for $k=3$ with the extra vertex $c^*=(k+1,n-k-1)$ (red). Dashed red lines: the
cuts between $\Pi_p$, $\cR$ and $\Pi'_p$. Squares: the two exterior vertices that become adjacent to two
pieces.}\label{fig:extra}
\end{figure}

\begin{lemma}\label{lem:extra}
Assume \eqref{eq:hyp}. For $\theta\in\{0,1\}$ and $w,w'\in W_s$ let
\[
C^{\cR'}_\theta[w,w']=\min\{\ell(D): D\subseteq\cR',\ I(\cR')\subseteq N[D],\ [c^*\in D]=\theta,\
\phi_D(\beta)=w,\ \phi_D(\sigma(\beta))=w'\},
\]
and let $L^{\theta}_s[w,w']=\lambda_s(w,w')+\theta\,[w_k=0]$ if $(w,w')$ is compatible and, in addition,
$w'_k\ne2$ or $w_k=0$; otherwise let $L^{\theta}_s[w,w']=+\infty$. Let $H'$ be the entrywise minimum of the
two matrices $L^{\theta}_s\ox C^{\cR'}_\theta\ox(L^{\theta}_s)^{T}$, $\theta\in\{0,1\}$. Then
\begin{equation}\label{eq:bprime}
b(B')=\min_{u,u'\in W_d}\bigl(A'_p[u,u']+G_m[u,u']\bigr),\qquad A'_p=C^{\cV}_p\ox H'\ox(C^{\cV}_p)^{T}.
\end{equation}
For $(k,e)=(10,6)$ we have $B'=B'_n$.
\end{lemma}
\begin{proof}
The sets $\Pi_p$, $\cR'$, $\Pi'_p$ and $\cD_m$ partition $B'$. The neighbors of $c^*$ are $\beta_k$,
$\sigma(\beta_k)$, $z_1$ and $z_2$; moreover $z_1$ is adjacent to $\alpha_k\in\Pi_p$ and
$z_2$ to $\sigma(\alpha_k)\in\Pi'_p$, and neither $z_1$ nor $z_2$ has another neighbor in $B'$. Hence the
hypothesis of Lemma~\ref{lem:glue} still holds, and with respect to Lemma~\ref{lem:structure} the only
changes are the following. First, $z_1$ is a vertex outside $B'$ with neighbors in $\Pi_p$ and $\cR'$, so
by Lemma~\ref{lem:glue}(c), with the notation of Lemma~\ref{lem:cuts},
$|N[D_1]\cap N[D_2]|=\lambda_s(w,w')+[\alpha_k\in D_1][c^*\in D_2]$. Second, $\beta_k$ now belongs to
$I(B')$ and its only neighbor outside $\cR'$ is $\alpha_k$, so it is dominated by $D$ if and only if
$w'_k\ne2$ or $w_k=0$; the vertices $\alpha_k$ and $c^*$ are not in $I(B')$, because $z_1\notin B'$. By
the symmetry $\sigma$, the same holds for the cut between $\cR'$ and $\Pi'_p$, with $z_2$,
$\sigma(\beta_k)$ and $\sigma(\alpha_k)$. The pieces $\Pi_p$, $\Pi'_p$, $\cD_m$ and the cuts between
$\cD_m$ and the acute corners are unchanged. The rest of the proof is that of Theorem~\ref{thm:glue},
splitting according to whether $c^*\in D$.
\end{proof}

The matrix $H'$ is computed with the dynamic program described above.
Now we are ready to conclude the proof.

\begin{proof}[Proof of Theorem~\ref{thm:periodic}]
By Propositions~\ref{prop:perD} and~\ref{prop:perS}, $G_{m+5}=G_m+5$ for $m\ge10$ and $A'_{p+1}=A'_p+2$
for $p\ge105$. If $n\ge135$, then $p\ge105$ and $m\ge98$, so by \eqref{eq:bprime}
$b(B'_{n+5})=b(B'_n)+15$. It therefore suffices to compute $b(B'_n)$ for five consecutive values of $n\ge135$;
Table~\ref{tab:periodic} gives ten of them, which were computed directly (with one pass through $\cR$
for $p=108$).
\end{proof}

\begin{table}[ht]
\centering
\begin{tabular}{r|rrrrrrrrrr}
\toprule
$n$ & 138 & 139 & 140 & 141 & 142 & 143 & 144 & 145 & 146 & 147\\\midrule
$b(B_n(10,6))$ & 395 & 396 & 400 & 402 & 405 & 410 & 411 & 415 & 417 & 420\\
$b(B'_n)$ & 396 & 397 & 401 & 403 & 406 & 411 & 412 & 416 & 418 & 421\\
$\nu(n)$ & 393 & 395 & 401 & 401 & 405 & 408 & 410 & 416 & 416 & 420\\
\bottomrule
\end{tabular}
\caption{Band bounds in the periodic regime, without and with the extra vertex.}\label{tab:periodic}
\end{table}

\section{Proof of Theorem~\ref{thm:lower}}\label{sec:lower}

\begin{proposition}[computer-assisted]\label{prop:finite}
\begin{enumerate}[label=\textup{(\alph*)},leftmargin=*]
\item[]
\item $\gamma(F_2(P_n))=d(n)$ for $13\le n\le33$.
\item $b\bigl(B_n(10,3)\cup\{(11,n-11)\}\bigr)\ge\nu(n)$ for $33\le n\le38$.
\item $b(B'_n)\ge\nu(n)$ for $39\le n\le137$.
\end{enumerate}
The values of \textup{(b)} and \textup{(c)} are listed in Table~\ref{tab:range} in Appendix A.
\end{proposition}

For (a), the dynamic program of Section~\ref{sec:periodic} is applied to the whole triangle $I(n)$: all
vertices are interior, there are no exterior vertices, and the result is $\ell_n$, whence $\gamma$ by
\eqref{eq:ln}. Processing by increasing $x+y$ the frontier has about $n/2+1$ cells. This reproduces the
values of~\cite{ALR} for $n\le25$ and gives $\gamma(F_2(P_n))=77,83,89,95,102,108,115,122$ for
$n=26,\dots,33$. For (b) and (c), formula \eqref{eq:bprime} is evaluated for each $n$, with
$C^{\cV}_p=C^{\cV}_0\ox T_s^{\ox p}$ and $(k,e,a)=(10,3,15)$ or $(10,6,18)$; the matrix $H'$ depends only
on $k$ and is computed once. In (c) the difference $b-\nu$ equals $0$ exactly for $n=50,55,\dots,135$.

\begin{proof}[Proof of Theorem~\ref{thm:lower}]
For $13\le n\le33$ this is Proposition~\ref{prop:finite}(a). For $n\ge33$, Lemma~\ref{lem:target} applies
with the band of Proposition~\ref{prop:finite}(b),(c) if $n\le137$ and with $B'_n$ and
Theorem~\ref{thm:periodic} if $n\ge135$. This proves Theorem~\ref{thm:lower}.
\end{proof}

\subsection{Computations and validation}

All computational tasks and validation were developed and executed using Python programs, supplemented by a specialized C routine for the \((\min,+)\) product. The complete process was run by Anthropic's Claude AI under the authors' supervision, concluding in different sessions but in less than three hours in total. 

{\bf AI Disclosure}: Anthropic's Claude AI was employed to develop and run the code, generate the figures, and polish the English grammar. The authors have audited the entire process. Also, the author verified all mathematical arguments and computations and takes full responsibility for the final content

\section*{Acknowledgments}
E. Acosta Troncoso is a PhD student of the Programa de Doctorado en Ciencias B\'asicas at Universidad Aut\'onoma de Zacatecas (UAZ) and has received a fellowship (Grant No. 4014068) from Secretar\'ia de Ciencia, Humanidades, Tecnolog\'ia e Innovaci\'on (SECIHTI), formerly Consejo Nacional de Humanidades, Ciencias y Tecnolog\'ias (CONAHCYT). 

\appendix
\section{Band bounds for \texorpdfstring{$33\le n\le137$}{33 <= n <= 137}}

\begin{table}[ht]
\centering\small
\begin{tabular}{rrr@{\qquad}rrr@{\qquad}rrr}
\toprule
$n$ & $b$ & $\nu(n)$ & $n$ & $b$ & $\nu(n)$ & $n$ & $b$ & $\nu(n)$\\\midrule
33 & 79 & 78 & 68 & 186 & 183 & 103 & 291 & 288\\
34 & 83 & 80 & 69 & 187 & 185 & 104 & 292 & 290\\
35 & 86 & 86 & 70 & 191 & 191 & 105 & 296 & 296\\
36 & 89 & 86 & 71 & 193 & 191 & 106 & 298 & 296\\
37 & 91 & 90 & 72 & 196 & 195 & 107 & 301 & 300\\
38 & 94 & 93 & 73 & 201 & 198 & 108 & 306 & 303\\
39 & 98 & 95 & 74 & 202 & 200 & 109 & 307 & 305\\
40 & 102 & 101 & 75 & 206 & 206 & 110 & 311 & 311\\
41 & 104 & 101 & 76 & 208 & 206 & 111 & 313 & 311\\
42 & 107 & 105 & 77 & 211 & 210 & 112 & 316 & 315\\
43 & 112 & 108 & 78 & 216 & 213 & 113 & 321 & 318\\
44 & 113 & 110 & 79 & 217 & 215 & 114 & 322 & 320\\
45 & 117 & 116 & 80 & 221 & 221 & 115 & 326 & 326\\
46 & 119 & 116 & 81 & 223 & 221 & 116 & 328 & 326\\
47 & 122 & 120 & 82 & 226 & 225 & 117 & 331 & 330\\
48 & 126 & 123 & 83 & 231 & 228 & 118 & 336 & 333\\
49 & 128 & 125 & 84 & 232 & 230 & 119 & 337 & 335\\
50 & 131 & 131 & 85 & 236 & 236 & 120 & 341 & 341\\
51 & 134 & 131 & 86 & 238 & 236 & 121 & 343 & 341\\
52 & 137 & 135 & 87 & 241 & 240 & 122 & 346 & 345\\
53 & 141 & 138 & 88 & 246 & 243 & 123 & 351 & 348\\
54 & 143 & 140 & 89 & 247 & 245 & 124 & 352 & 350\\
55 & 146 & 146 & 90 & 251 & 251 & 125 & 356 & 356\\
56 & 149 & 146 & 91 & 253 & 251 & 126 & 358 & 356\\
57 & 152 & 150 & 92 & 256 & 255 & 127 & 361 & 360\\
58 & 156 & 153 & 93 & 261 & 258 & 128 & 366 & 363\\
59 & 158 & 155 & 94 & 262 & 260 & 129 & 367 & 365\\
60 & 161 & 161 & 95 & 266 & 266 & 130 & 371 & 371\\
61 & 164 & 161 & 96 & 268 & 266 & 131 & 373 & 371\\
62 & 166 & 165 & 97 & 271 & 270 & 132 & 376 & 375\\
63 & 171 & 168 & 98 & 276 & 273 & 133 & 381 & 378\\
64 & 172 & 170 & 99 & 277 & 275 & 134 & 382 & 380\\
65 & 176 & 176 & 100 & 281 & 281 & 135 & 386 & 386\\
66 & 178 & 176 & 101 & 283 & 281 & 136 & 388 & 386\\
67 & 181 & 180 & 102 & 286 & 285 & 137 & 391 & 390\\
\bottomrule
\end{tabular}
\caption{The bound $b$ of Proposition~\ref{prop:finite}: $(k,e,a)=(10,3,15)$ for $33\le n\le38$ and
$(10,6,18)$ for $39\le n\le137$, with the extra vertex in both cases.}\label{tab:range}
\end{table}

\end{document}